\documentclass[11pt]{article}

\usepackage[margin=1in]{geometry}
\usepackage{amsmath,amssymb,amsthm,mathtools}
\usepackage{microtype}
\usepackage[colorlinks=true,linkcolor=blue,citecolor=blue,urlcolor=blue]{hyperref}

\hypersetup{
  pdftitle={Radial Extremality for LRU Caching and the Fill--Holst Conjecture},
  pdfauthor={Christopher D. Long},
  pdfsubject={LRU caching and the Fill--Holst conjecture}
}

\newtheorem{theorem}{Theorem}[section]
\newtheorem{lemma}[theorem]{Lemma}

\newtheorem{corollary}[theorem]{Corollary}
\theoremstyle{remark}
\newtheorem{remark}[theorem]{Remark}

\newcommand{\R}{\mathbb{R}}

\newcommand{\dd}{\,\mathrm{d}}
\newcommand{\E}{\mathbb{E}}
\newcommand{\Pp}{\mathbb{P}}
\newcommand{\Exp}{\operatorname{Exp}}
\newcommand{\card}[1]{\lvert #1\rvert}
\newcommand{\set}[1]{\left\{#1\right\}}
\newcommand{\one}{\mathbf{1}}
\newcommand{\doi}[1]{\href{https://doi.org/#1}{\nolinkurl{doi:#1}}}

\title{Radial Extremality for LRU Caching and the Fill--Holst Conjecture}
\author{Christopher D. Long\\
Headlamp Software\\
\texttt{galizur@gmail.com}}
\date{}

\begin{document}
\maketitle

\begin{abstract}
For the independent reference model with popularity vector
$p\in\Delta_N^\circ$, let $H_C(p)$ denote the exact stationary hit rate
of an LRU cache of capacity $C$.  We prove that, for every $1\le C<N$,
the uniform popularity vector is the unique global minimizer of $H_C$ on
the interior simplex.  More sharply, along every nonconstant segment from
the uniform vector to an interior point, the LRU hit rate is strictly
increasing.  The proof uses the standard exponential-age representation
of the stationary LRU cache and gives an explicit positive pair-square
formula for the radial derivative.  Equivalently, for the move-to-front
rule, the stationary search-cost distribution improves strictly in the
usual stochastic order along every nonconstant ray away from uniform.
This proves the radial restriction of the Fill--Holst Schur-concavity
conjecture for move-to-front search-cost tails.  In particular, all LRU
miss probabilities and all nonconstant nondecreasing stack-depth costs
decrease strictly along such rays.  The result is radial
rather than Schur-convex: full majorization monotonicity for LRU is known
to fail, and the proof identifies the special positivity that survives on
rays from the uniform vector.
\end{abstract}

\section{Introduction}

The independent reference model (IRM) is the classical model in which
successive requests are independent and item $i$ is requested with
probability $p_i$ in discrete time, equivalently with rate $p_i$ in the
rate-one Poissonized model.  Under LRU, the stationary stack order has an
exact exponential-age representation: if $A_1,\ldots,A_N$ are independent
exponential random variables with $A_i\sim\Exp(p_i)$, then the stationary
cache of capacity $C$ consists of the $C$ items with smallest ages.  This
representation follows by Poissonizing the request sequence into
independent Poisson streams of rates $p_i$ and using their backward
recurrence times; conditioning on the next arrival recovers the original
discrete request distribution.  The hit rate is therefore an exact
finite-dimensional function of the popularity vector, not a
characteristic-time or Che-type approximation.
Exact formulae and asymptotic approximations for LRU under the IRM go
back to King, Fagin, Fagin--Price, and Flajolet--Gardy--Thimonier
\cite{King1971,Fagin1977,FaginPrice1978,FlajoletGardyThimonier1992};
Che--Tung--Wang introduced a now-standard characteristic-time
approximation for hierarchical Web caching \cite{CheTungWang2002},
Jelenkovi\'c studied asymptotic MTF/LRU fault probabilities for broad
popularity classes \cite{Jelenkovic1999}, and Berthet discusses the
identity of the King and Flajolet--Gardy--Thimonier formulae for exact
LRU miss-rate computation \cite{Berthet2016}.

A natural monotonicity question asks whether making the popularity vector
more uneven must improve the LRU hit rate.  In the language of
majorization \cite{MarshallOlkinArnold2011}, this would assert
Schur-convexity of the hit rate, or equivalently Schur-concavity of the
miss rate.  A closely related move-to-front literature considered the
same distributional question in search-cost language.  Under independent
requests, the stationary move-to-front list is the LRU recency order, and
the LRU miss probability for a cache of capacity $C$ is the tail
probability that the move-to-front search cost exceeds $C$.  Fill and
Holst studied the stationary move-to-front search-cost distribution and
its cache-fault interpretation, and Hildebrand later addressed their
conjecture that each search-cost tail is Schur-concave in the popularity
vector \cite{FillHolst1996,Hildebrand1999}.  Hildebrand disproved that
conjecture.

Thus full Schur-concavity is too strong even for the natural MTF/LRU tail
probabilities.  Vanichpun and Makowski also showed, in a broader caching
setting, that majorization monotonicity may fail under LRU
\cite{VanichpunMakowski2004,MakowskiVanichpun2005}.  The theorem below
identifies a sharp positive remnant of that false global principle:
although arbitrary majorization directions can move the LRU hit rate in
the wrong direction, every straight-line move away from the uniform
popularity vector strictly increases the hit rate.  Equivalently, the
radial restriction of the Fill--Holst tail-concavity conjecture is true
for MTF/LRU\@.

The proof is deliberately finite and exact.  We first derive a residual
subset expansion for $H_C(p)$.  A variance identity on each residual
subset converts that expansion into a pair-square decomposition around
the uniform vector.  Radial differentiation then produces an explicit
pair-square formula
\[
        \frac{\dd}{\dd\theta}H_C(u+\theta(q-u))
        =
        \frac{1}{N\theta}
        \sum_{1\le a<b\le N}
        (p_a(\theta)-p_b(\theta))^2 K_{ab}^{(C)}(p(\theta)),
\]
where every kernel $K_{ab}^{(C)}$ is strictly positive.  The positivity
of the kernels is proved by converting the alternating subset sums into
positive product integrals.  This gives a local certificate for radial
monotonicity, not merely a global comparison.

After proving the main theorem, we record its distributional
consequences for the full LRU stack-depth distribution, or equivalently
for the stationary move-to-front search-cost distribution.  We also
include an appendix with an independent occupancy-Jacobian proof.  That
proof explains at the differential level why radial monotonicity is
compatible with failure of Schur-convexity: the asymmetric occupancy
Jacobian is not positive in every tangent direction, but its
sign-indefinite minor collapses to a positive square along rays from the
uniform vector.

\section{Setup and main results}

Let
\[
        \Delta_N^\circ=
        \set{p=(p_1,\ldots,p_N):p_i>0,\ \sum_{i=1}^N p_i=1},
        \qquad
        u=\left(\frac1N,\ldots,\frac1N\right).
\]
For $R\subseteq [N]:=\set{1,\ldots,N}$, write
\[
        p_R=\sum_{i\in R}p_i,
        \qquad p_\varnothing=0.
\]
Let $A_1,\ldots,A_N$ be independent exponential random variables with
$A_i\sim\Exp(p_i)$, and let $S_C$ be the set of the $C$ indices with
smallest ages.  The exact stationary LRU hit rate is
\begin{equation}\label{eq:HC-def}
        H_C(p)=\E\left[\sum_{i\in S_C}p_i\right].
\end{equation}
Equivalently, since the next request has distribution $p$ independently
of the stationary stack, $H_C(p)$ is the stationary probability that the
next requested item is already in an LRU cache of capacity $C$.

\begin{theorem}[Radial LRU extremality]\label{thm:main}
Let $N\ge2$, let $1\le C<N$, and let $q\in\Delta_N^\circ$ be nonuniform.
Define
\[
        p(\theta)=u+\theta(q-u),\qquad 0\le\theta\le1.
\]
Then
\[
        \frac{\dd}{\dd\theta}H_C(p(\theta))>0,
        \qquad 0<\theta\le1.
\]
Consequently $\theta\mapsto H_C(p(\theta))$ is strictly increasing on
$(0,1]$.
\end{theorem}

\begin{corollary}[Uniform popularity is the unique worst case]\label{cor:global-minimum}
Let $N\ge2$ and $1\le C<N$.  For every nonuniform
$p\in\Delta_N^\circ$,
\[
        H_C(p)>H_C(u)=\frac{C}{N}.
\]
Thus $u$ is the unique global minimizer of the LRU hit rate on
$\Delta_N^\circ$.  Equivalently, the LRU miss rate $1-H_C(p)$ is uniquely
maximized at $u$ on $\Delta_N^\circ$.
\end{corollary}

\begin{proof}
Apply Theorem~\ref{thm:main} with $q=p$.  Then $p(0)=u$ and $p(1)=p$.
By Lemma~\ref{lem:residual}, $H_C$ is continuous along this segment, and
Theorem~\ref{thm:main} gives strict monotonicity on $(0,1]$; hence
$H_C(p)>H_C(u)$.  Under $u$, symmetry gives $\Pp(i\in S_C)=C/N$ for
every item $i$, and therefore
\[
        H_C(u)=\sum_{i=1}^N \frac1N\cdot\frac{C}{N}=\frac{C}{N}.
\]
\end{proof}

The proof of Theorem~\ref{thm:main} proceeds through a stronger positive
kernel formula.  Put
\[
        L=N-C+1.
\]
Thus $2\le L\le N$.  For $a<b$ and $p\in\Delta_N^\circ$, define
\begin{equation}\label{eq:kernel-def}
        K_{ab}^{(C)}(p)
        =
        N
        \sum_{\substack{R\subseteq[N]:\ \{a,b\}\subseteq R\\ \card R\ge L}}
        (-1)^{\card R-L}
        \binom{\card R-2}{\card R-L}
        \left(
            \frac{1}{\card R\,p_R}
            +\frac{1}{N p_R^2}
        \right).
\end{equation}

\begin{theorem}[Positive pair-kernel derivative formula]\label{thm:kernel}
For every $p\in\Delta_N^\circ$ and every $a<b$,
\[
        K_{ab}^{(C)}(p)>0.
\]
Moreover, if $p(\theta)=u+\theta(q-u)$ with $0<\theta\le1$, then
\begin{equation}\label{eq:radial-derivative}
        \frac{\dd}{\dd\theta}H_C(p(\theta))
        =
        \frac{1}{N\theta}
        \sum_{1\le a<b\le N}
        \bigl(p_a(\theta)-p_b(\theta)\bigr)^2
        K_{ab}^{(C)}(p(\theta)).
\end{equation}
\end{theorem}

Theorem~\ref{thm:main} follows immediately from Theorem~\ref{thm:kernel}:
if $q\ne u$ and $0<\theta\le1$, then $p(\theta)$ is nonuniform, so at
least one pair difference in \eqref{eq:radial-derivative} is nonzero.

\section{A residual-subset formula for the hit rate}

We first convert the exponential-clock representation into a finite
alternating subset expansion.  This is the only place where the cache
capacity enters in an essential way.

\begin{lemma}[Residual-subset expansion]\label{lem:residual}
Let $1\le C<N$ and $L=N-C+1$.  For $L\le m\le N$, set
\[
        \alpha_m=(-1)^{m-L}\binom{m-2}{m-L}.
\]
Then, for every $p\in\Delta_N^\circ$,
\begin{equation}\label{eq:residual-formula}
        H_C(p)
        =
        \sum_{m=L}^N \alpha_m
        \sum_{\substack{R\subseteq[N]\\ \card R=m}}
        \frac{\sum_{i\in R}p_i^2}{p_R}.
\end{equation}
\end{lemma}

\begin{proof}
Fix an item $i$.  Its contribution to the hit rate is $p_i\Pp(i\in S_C)$.
Conditioning on $A_i=t$ gives
\begin{align}
        p_i\Pp(i\in S_C)
        & =
        p_i^2\int_0^\infty e^{-p_i t}
        \sum_{\substack{S\subseteq [N]\setminus\{i\}\\ \card S\le C-1}}
        \prod_{j\in S}(1-e^{-p_jt})
        \prod_{\ell\notin S\cup\{i\}}e^{-p_\ell t}
        \dd t.
        \label{eq:item-contribution}
\end{align}
Here $S$ is the set of other items whose ages are smaller than $t$; item
$i$ is in the cache precisely when $\card S\le C-1$.

Expand
\[
        \prod_{j\in S}(1-e^{-p_jt})
        =
        \sum_{A\subseteq S}(-1)^{\card A}e^{-p_A t}.
\]
For a term indexed by $A\subseteq S$, put $B=S\setminus A$.  The total
exponent in \eqref{eq:item-contribution} is
\[
        p_i+p_A+p_{[N]\setminus(S\cup\{i\})}=1-p_B.
\]
Now fix $B\subseteq [N]\setminus\{i\}$ with $b=\card B\le C-1$.  The set
$A$ may be any subset of $[N]\setminus(\{i\}\cup B)$ with
$\card A\le C-1-b$.  Hence the coefficient of $e^{-(1-p_B)t}$ is
\[
        \sum_{h=0}^{C-1-b}(-1)^h\binom{N-1-b}{h}.
\]
Since $C<N$, the elementary identity
\[
        \sum_{h=0}^q(-1)^h\binom{M}{h}
        =(-1)^q\binom{M-1}{q},
        \qquad 0\le q<M,
\]
shows that this coefficient is
\[
        (-1)^{C-1-b}\binom{N-2-b}{C-1-b}.
\]
Therefore
\[
        p_i\Pp(i\in S_C)
        =
        p_i^2
        \sum_{\substack{B\subseteq[N]\setminus\{i\}\\ \card B\le C-1}}
        (-1)^{C-1-\card B}
        \binom{N-2-\card B}{C-1-\card B}
        \frac{1}{1-p_B}.
\]
Set $R=[N]\setminus B$.  Then $i\in R$, $p_R=1-p_B$, and
$m:=\card R=N-\card B$ satisfies $m\ge N-C+1=L$.  Moreover
\[
        (-1)^{C-1-\card B}
        \binom{N-2-\card B}{C-1-\card B}
        =
        (-1)^{m-L}\binom{m-2}{m-L}=\alpha_m.
\]
Summing over $i$ and grouping by $R$ gives \eqref{eq:residual-formula}.
\end{proof}

\section{Pair-square decomposition}

The residual expansion separates cleanly into the uniform contribution
and pairwise fluctuations.

\begin{lemma}[Pair-square expansion of $H_C$]\label{lem:pair-square}
For every $p\in\Delta_N^\circ$,
\begin{equation}\label{eq:pair-square-H}
        H_C(p)
        =
        \frac{C}{N}
        +
        \sum_{1\le a<b\le N}(p_a-p_b)^2J_{ab}^{(C)}(p),
\end{equation}
where
\begin{equation}\label{eq:J-def}
        J_{ab}^{(C)}(p)
        =
        \sum_{\substack{R\subseteq[N]:\ \{a,b\}\subseteq R\\ \card R\ge L}}
        (-1)^{\card R-L}
        \binom{\card R-2}{\card R-L}
        \frac{1}{\card R\,p_R}.
\end{equation}
\end{lemma}

\begin{proof}
For any nonempty finite set $R$ with $m=\card R$,
\begin{equation}\label{eq:variance-identity}
        \sum_{i\in R}p_i^2
        =
        \frac{p_R^2}{m}
        +
        \frac{1}{m}
        \sum_{\substack{a<b\\ a,b\in R}}(p_a-p_b)^2.
\end{equation}
This follows from
\[
        \sum_{\substack{a<b\\ a,b\in R}}(p_a-p_b)^2
        =m\sum_{i\in R}p_i^2-p_R^2.
\]
Substituting \eqref{eq:variance-identity} into
\eqref{eq:residual-formula} gives
\begin{align*}
        H_C(p)
        &=
        \sum_{m=L}^N\alpha_m
        \sum_{\substack{R\subseteq[N]\\ \card R=m}}
        \frac{p_R}{m} \\
        &\quad+
        \sum_{1\le a<b\le N}(p_a-p_b)^2
        \sum_{\substack{R\subseteq[N]:\ \{a,b\}\subseteq R\\ \card R\ge L}}
        \frac{\alpha_{\card R}}{\card R\,p_R}.
\end{align*}
The second term is the pair-square term in \eqref{eq:pair-square-H}.  The
first term is constant on the simplex, because
\[
        \sum_{\substack{R\subseteq[N]\\ \card R=m}}p_R
        =
        \binom{N-1}{m-1}\sum_{i=1}^Np_i
        =
        \binom{N-1}{m-1}.
\]
To identify the constant, evaluate at $p=u$.  By symmetry,
$H_C(u)=C/N$, and at $p=u$ all pair differences vanish.  Hence the
constant is $C/N$.
\end{proof}

\section{Radial differentiation}

Let $q\in\Delta_N^\circ$ and set
\[
        p(\theta)=u+\theta(q-u),\qquad 0<\theta\le1.
\]
For readability, write $x=p(\theta)$ during the differentiation.  If
$R\subseteq[N]$ and $m=\card R$, then
\begin{equation}\label{eq:radial-basic-derivatives}
        \theta\frac{\dd x_i}{\dd\theta}=x_i-\frac1N,
        \qquad
        \theta\frac{\dd x_R}{\dd\theta}=x_R-\frac{m}{N},
        \qquad
        \theta\frac{\dd}{\dd\theta}(x_a-x_b)=x_a-x_b.
\end{equation}
Consequently,
\begin{align}
        \theta\frac{\dd}{\dd\theta}
        \left(\frac{(x_a-x_b)^2}{m x_R}\right)
        &=(x_a-x_b)^2
        \left(
            \frac{2}{m x_R}
            -
            \frac{x_R-m/N}{m x_R^2}
        \right) \notag\\
        &=(x_a-x_b)^2
        \left(
            \frac{1}{m x_R}
            +
            \frac{1}{N x_R^2}
        \right).
        \label{eq:term-differentiation}
\end{align}
Differentiating \eqref{eq:pair-square-H} term by term and using
\eqref{eq:term-differentiation} gives
\[
        \frac{\dd}{\dd\theta}H_C(p(\theta))
        =
        \frac{1}{\theta}
        \sum_{1\le a<b\le N}
        (x_a-x_b)^2
        \sum_{\substack{R\subseteq[N]:\ \{a,b\}\subseteq R\\ \card R\ge L}}
        (-1)^{\card R-L}
        \binom{\card R-2}{\card R-L}
        \left(
            \frac{1}{\card R\,x_R}
            +
            \frac{1}{N x_R^2}
        \right).
\]
This is precisely \eqref{eq:radial-derivative}, with $K_{ab}^{(C)}$
defined by \eqref{eq:kernel-def}.  It remains only to prove that these
kernels are strictly positive.

\section{Strict positivity of the pair kernels}

Fix $p\in\Delta_N^\circ$ and a pair $a<b$.  Let
\[
        T=[N]\setminus\{a,b\},
        \qquad
        s=p_a+p_b,
        \qquad
        r=L-2=N-C-1.
\]
Since $1\le C<N$, we have $0\le r\le N-2=\card T$.  Writing
$R=\{a,b\}\cup U$ with $U\subseteq T$, the coefficient in
\eqref{eq:kernel-def} becomes
\[
        (-1)^{\card R-L}\binom{\card R-2}{\card R-L}
        =
        (-1)^{\card U-r}\binom{\card U}{r},
\]
where the binomial coefficient is zero if $\card U<r$.  Hence
\begin{equation}\label{eq:K-Phi-Psi}
        \frac1N K_{ab}^{(C)}(p)=\Phi_r+\frac1N\Psi_r,
\end{equation}
where
\begin{align}
        \Phi_r
        &: =
        \sum_{U\subseteq T}
        (-1)^{\card U-r}\binom{\card U}{r}
        \frac{1}{(\card U+2)(s+p_U)},
        \label{eq:Phi-def}
        \\
        \Psi_r
        &: =
        \sum_{U\subseteq T}
        (-1)^{\card U-r}\binom{\card U}{r}
        \frac{1}{(s+p_U)^2}.
        \label{eq:Psi-def}
\end{align}
We prove that both $\Phi_r$ and $\Psi_r$ are strictly positive.

For $t\ge0$ and $0\le y\le1$, define
\begin{equation}\label{eq:B-def}
        B_r(y,t)
        =
        \sum_{U\subseteq T}
        (-1)^{\card U-r}\binom{\card U}{r}
        y^{\card U}e^{-p_Ut}.
\end{equation}
Let
\[
        P_t(y)=\prod_{j\in T}(1-y e^{-p_jt}).
\]
Since
\[
        P_t(y)=\sum_{U\subseteq T}(-1)^{\card U}y^{\card U}e^{-p_Ut},
\]
we have
\begin{equation}\label{eq:B-derivative}
        B_r(y,t)=\frac{(-1)^r y^r}{r!}\,P_t^{(r)}(y).
\end{equation}
Differentiating the product gives the positive form
\begin{equation}\label{eq:B-positive-form}
        B_r(y,t)
        =
        y^r
        \sum_{\substack{W\subseteq T\\ \card W=r}}
        e^{-p_Wt}
        \prod_{j\in T\setminus W}(1-y e^{-p_jt}).
\end{equation}
For $t>0$ and $0<y\le1$, every factor
$1-y e^{-p_jt}$ is strictly positive, because $p_j>0$.  Therefore
\begin{equation}\label{eq:B-positive}
        B_r(y,t)>0,
        \qquad t>0,
        \quad 0<y\le1.
\end{equation}

Now use the elementary integral identities
\[
        \frac{1}{s+p_U}
        =\int_0^\infty e^{-(s+p_U)t}\dd t,
        \qquad
        \frac{1}{\card U+2}
        =\int_0^1 y^{\card U+1}\dd y,
\]
and
\[
        \frac{1}{(s+p_U)^2}
        =\int_0^\infty t e^{-(s+p_U)t}\dd t.
\]
Since all sums are finite, we may interchange sums and integrals.  From
\eqref{eq:Phi-def} and \eqref{eq:B-def},
\[
        \Phi_r
        =
        \int_0^\infty e^{-st}
        \int_0^1 y B_r(y,t)\dd y\dd t.
\]
By \eqref{eq:B-positive}, the integrand is strictly positive for
$t>0$ and $0<y<1$.  Hence $\Phi_r>0$.  Similarly, from
\eqref{eq:Psi-def} and \eqref{eq:B-def},
\[
        \Psi_r
        =
        \int_0^\infty t e^{-st}B_r(1,t)\dd t>0.
\]
Combining these inequalities with \eqref{eq:K-Phi-Psi} gives
$K_{ab}^{(C)}(p)>0$.  This proves Theorem~\ref{thm:kernel}, and hence
Theorem~\ref{thm:main}.

\begin{remark}[Endpoint capacities]
For $C=1$, one has $L=N$, and the residual formula reduces to
\[
        H_1(p)=\sum_{i=1}^N p_i^2.
\]
For $C=N$, the hit rate is identically $1$, so strict radial monotonicity
is not a meaningful statement.  The theorem therefore covers exactly the
nontrivial range $1\le C<N$.
\end{remark}

\section{Move-to-front search costs and miss-probability consequences}\label{sec:stack-depth}

The radial theorem has a useful distributional interpretation.  Under
independent requests, the stationary move-to-front list is the same
recency ordering used by LRU, namely the ordering of all $N$ items from
most recently requested to least recently requested.  Let $D=D(p)$ be the
stationary move-to-front search cost, equivalently the LRU stack depth of
the next requested item in stationarity.  Then $D\in\{1,\ldots,N\}$, and a
cache of capacity $C$ hits exactly when $D\le C$.  Therefore
\begin{equation}\label{eq:stack-depth-cdf}
        \Pp_p(D\le C)=H_C(p),
        \qquad 1\le C\le N.
\end{equation}
Here $H_N\equiv1$, while the nontrivial thresholds are
$1\le C<N$.

\begin{corollary}[Radial stochastic improvement for move-to-front search cost]\label{cor:stack-depth}
Let $N\ge2$, let $q\in\Delta_N^\circ$ be nonuniform, and put
$p(\theta)=u+\theta(q-u)$.  If $0<\theta_1<\theta_2\le1$, then the
move-to-front search cost under $p(\theta_2)$ is strictly smaller than the
move-to-front search cost under $p(\theta_1)$ in the usual stochastic
order: for every $1\le C<N$,
\[
        \Pp_{p(\theta_2)}(D\le C)
        >
        \Pp_{p(\theta_1)}(D\le C),
\]
or equivalently
\[
        \Pp_{p(\theta_2)}(D>C)
        <
        \Pp_{p(\theta_1)}(D>C).
\]
\end{corollary}

\begin{proof}
Identity \eqref{eq:stack-depth-cdf} gives
$\Pp_{p(\theta)}(D\le C)=H_C(p(\theta))$.  For $1\le C<N$,
Theorem~\ref{thm:main} shows that this quantity is strictly increasing in
$\theta$.  Taking complements gives the tail form.
\end{proof}

\begin{corollary}[Nondecreasing search-cost functionals]\label{cor:increasing-costs}
Under the hypotheses of Corollary~\ref{cor:stack-depth}, for every
nondecreasing function $g:\{1,\ldots,N\}\to\R$,
\[
        \E_{p(\theta_2)}g(D)\le \E_{p(\theta_1)}g(D).
\]
The inequality is strict whenever $g$ is not constant.  In particular, the
LRU miss probability
\[
        M_C(p):=1-H_C(p)=\Pp_p(D>C)
\]
is strictly decreasing along every nonconstant ray away from uniform for
each $1\le C<N$, and the expected move-to-front search cost
\[
        \E_p D=1+\sum_{c=1}^{N-1}M_c(p)
\]
is strictly decreasing along every such ray.
\end{corollary}

\begin{proof}
Use the discrete tail-sum decomposition
\begin{equation}\label{eq:g-tail-decomp}
        \E g(D)=g(1)+\sum_{c=1}^{N-1}\bigl(g(c+1)-g(c)\bigr)\Pp(D>c).
\end{equation}
If $g$ is nondecreasing, every coefficient in the sum is nonnegative, and
Corollary~\ref{cor:stack-depth} gives the asserted monotonicity.  If
$g$ is not constant, at least one coefficient $g(c+1)-g(c)$ is positive,
and the corresponding tail probability decreases strictly.  The statements
about miss probability and expected search depth are the special cases
$g=\one_{\{C+1,\ldots,N\}}$ and $g(d)=d$.
\end{proof}

In this language, Theorem~\ref{thm:main} proves the radial restriction of
the Fill--Holst tail-concavity conjecture.  Hildebrand's counterexample
shows that full Schur-concavity of the search-cost tails is false.  The
present result shows that the conjectured monotonicity nevertheless
survives on every segment emanating from the uniform distribution.

\section{Why the theorem is radial, not Schur-convex}\label{sec:not-schur}

The preceding result should not be interpreted as a disguised
Schur-convexity theorem.  Full Schur-convexity would require monotonicity
under all elementary majorization transfers, not merely along rays from the
uniform vector.  The Fill--Holst tail conjecture is false
\cite{Hildebrand1999}, and the known Vanichpun--Makowski examples show
that LRU can violate the corresponding majorization monotonicity
\cite{VanichpunMakowski2004,MakowskiVanichpun2005}.
The point is that the radial direction from uniform aligns with a local
asymmetry in the LRU occupancy map.

Appendix~\ref{app:jacobian-proof} gives the details, but the local
mechanism is simple to state.  Let $\pi_k(\lambda)$ denote the probability
that item $k$ is present in the stationary LRU cache when the exponential
rates are $\lambda$.  For $i\ne k$, the off-diagonal sensitivity has the
form
\[
        \frac{\partial \pi_k}{\partial \lambda_i}(\lambda)
        =-\lambda_k G_{ik}(\lambda),
        \qquad G_{ik}(\lambda)=G_{ki}(\lambda)>0.
\]
Thus the occupancy Jacobian is generally asymmetric.  For an arbitrary
tangent direction $\delta$ with $\sum_i\delta_i=0$, the occupancy-sensitivity
part of the directional derivative contains
\[
        \sum_{i<k}G_{ik}(\lambda)
        (\lambda_i-\lambda_k)(\lambda_k\delta_i-\lambda_i\delta_k),
\]
and the minor $\lambda_k\delta_i-\lambda_i\delta_k$ has no fixed sign.
Along the radial path $\lambda=u+\theta(q-u)$, however,
\[
        \delta_i=\frac{\lambda_i-1/N}{\theta},
        \qquad
        \lambda_k\delta_i-\lambda_i\delta_k
        =\frac{\lambda_i-\lambda_k}{N\theta},
\]
so the sign-indefinite term collapses to the positive square
\[
        \frac{1}{N\theta}
        \sum_{i<k}G_{ik}(\lambda)(\lambda_i-\lambda_k)^2.
\]
This is the local reason radial monotonicity survives even though full
Schur-convexity fails.

\section{Further directions}\label{sec:further-directions}

The distinction between global Schur-concavity and radial monotonicity
suggests a broader structural problem.  For policies such as RANDOM
replacement and FIFO, stronger Schur-concavity results are already known
under the IRM, so radial monotonicity follows immediately
\cite{VanichpunMakowski2004,MakowskiVanichpun2005}.  LRU/MTF is different:
full Schur-concavity fails, but the present paper proves strict radial
monotonicity.  This raises the corresponding radial question for other
self-organizing policies for which full Schur-concavity fails or remains
unknown.  Identifying the finite-state structure that controls radial
monotonicity for such policies is left for future work.

\appendix

\section{An occupancy-Jacobian proof}\label{app:jacobian-proof}

This appendix gives an independent proof of Theorem~\ref{thm:main}.  The
main proof above is shorter, but the Jacobian proof isolates the local
mechanism behind the radial positivity.

Fix $\lambda\in\R^N_{>0}$, and let
$A_i\sim\Exp(\lambda_i)$ be independent.  Let $S_C(\lambda)$ be the set
of the $C$ indices with smallest ages, and define
\[
        \pi_k(\lambda)=\Pp(k\in S_C(\lambda)).
\]
The map $\pi_k$ is homogeneous of degree zero, and
\begin{equation}\label{eq:app-sums}
        \sum_{k=1}^N\pi_k(\lambda)=C,
        \qquad
        H_C(p)=\sum_{k=1}^Np_k\pi_k(p).
\end{equation}
The first identity holds because $\card S_C=C$ deterministically.
Differentiating it gives
\begin{equation}\label{eq:app-rowsum}
        \sum_{k=1}^N\frac{\partial\pi_k}{\partial\lambda_i}(\lambda)=0
        \qquad\text{for every }i.
\end{equation}

\begin{lemma}[Occupancy sensitivity]\label{lem:app-sens}
For $i\ne k$ and any $\lambda\in\R^N_{>0}$,
\begin{equation}\label{eq:app-sens}
        \frac{\partial\pi_k}{\partial\lambda_i}(\lambda)
        =-\lambda_k G_{ik}(\lambda),
\end{equation}
where
\[
        G_{ik}(\lambda)
        =
        \int_0^\infty t e^{-(\lambda_i+\lambda_k)t}
        \Pp\bigl(M_{ik}(t)=C-1\bigr)\dd t,
\]
and
\[
        M_{ik}(t)=\#\{j\notin\{i,k\}:A_j<t\}.
\]
The kernel is symmetric and nonnegative, $G_{ik}=G_{ki}\ge0$.  If
$1\le C<N$, then $G_{ik}>0$ for every pair $i\ne k$.
\end{lemma}

\begin{proof}
Conditioning on $A_k=t$ gives
\[
        \pi_k(\lambda)
        =
        \int_0^\infty \lambda_k e^{-\lambda_k t}\Phi_k(t;\lambda)\dd t,
\]
where
\[
        \Phi_k(t;\lambda)
        =
        \Pp\left(\sum_{j\ne k}\one_{\{A_j<t\}}\le C-1\right).
\]
For $i\ne k$, write
$M_{ik}(t)=\sum_{j\notin\{i,k\}}\one_{\{A_j<t\}}$.  Splitting off the
$i$th indicator and putting $x_i(t)=1-e^{-\lambda_i t}$,
\[
        \Phi_k
        =\Pp(M_{ik}\le C-1)-x_i(t)\Pp(M_{ik}=C-1).
\]
Hence
\[
        \frac{\partial\Phi_k}{\partial\lambda_i}
        =-t e^{-\lambda_i t}\Pp(M_{ik}(t)=C-1).
\]
Differentiating under the integral is justified by dominated convergence,
for example on compact subsets of $\R^N_{>0}$: the differentiated
integrand is bounded by an integrable multiple of $t e^{-\eta t}$ for
some $\eta>0$.  This gives \eqref{eq:app-sens}.  The symmetry
$G_{ik}=G_{ki}$ is immediate from the integral expression.
For strict positivity, note that when $1\le C<N$, we have
$0\le C-1\le N-2$.  For every $t>0$, each event probability in the
Poisson-binomial mass $\Pp(M_{ik}(t)=C-1)$ is a nonempty sum of strictly
positive products, so the integrand is strictly positive on $(0,\infty)$.
\end{proof}

The diagonal entries are determined by \eqref{eq:app-rowsum}:
\begin{equation}\label{eq:app-diagonal}
        \frac{\partial\pi_i}{\partial\lambda_i}(\lambda)
        =\sum_{k\ne i}\lambda_kG_{ik}(\lambda).
\end{equation}

We also need the elementary monotonicity of occupancy probabilities.

\begin{lemma}[Ordering law]\label{lem:app-ordering-law}
For every permutation $\sigma=(\sigma_1,\ldots,\sigma_N)$ of $[N]$,
\[
        \Pp(A_{\sigma_1}<A_{\sigma_2}<\cdots<A_{\sigma_N})
        =
        \prod_{r=1}^N
        \frac{\lambda_{\sigma_r}}{\sum_{m=r}^N\lambda_{\sigma_m}}.
\]
Equivalently, the exponential-clock order is generated by repeatedly
choosing the next index from the remaining indices with probability
proportional to its rate.
\end{lemma}

\begin{proof}
The first clock to ring has label $j$ with probability
$\lambda_j/\sum_m\lambda_m$.  Conditional on that first label, the
residual waiting times of the remaining clocks are again independent
exponentials with their original rates, by memorylessness.  Iterating
gives the displayed product.
\end{proof}

\begin{lemma}[Occupancy monotonicity]\label{lem:app-mono}
If $\lambda_k\ge\lambda_i$, then $\pi_k(\lambda)\ge\pi_i(\lambda)$.
Consequently,
\[
        \left(\lambda_a-\lambda_b\right)
        \left(\pi_a(\lambda)-\pi_b(\lambda)\right)\ge0
        \qquad\text{for all }a,b.
\]
\end{lemma}

\begin{proof}
It suffices to compare the cases in which exactly one of $i,k$ belongs to
the first $C$ positions of the exponential-clock order, since the cases in
which both are cached or both are not cached cancel in $\pi_k-\pi_i$.
Let $\mathcal A$ be the set of orderings in which $i$ appears in one of
the first $C$ positions and $k$ appears after position $C$, and let
$\mathcal B$ be the analogous set with $i$ and $k$ interchanged.  Swapping
the labels $i$ and $k$ gives a bijection $\mathcal A\to\mathcal B$.

Fix $\sigma\in\mathcal A$.  Suppose $i$ occurs at position $r\le C$ and
$k$ occurs at position $s>C$.  Let $\tau$ be obtained from $\sigma$ by
swapping the labels $i$ and $k$, and put $d=\lambda_k-\lambda_i\ge0$.
In the product formula of Lemma~\ref{lem:app-ordering-law}, the numerator
factors involving $i$ and $k$ cancel in the ratio
$\Pp(\tau)/\Pp(\sigma)$, and all denominator factors outside positions
$r+1,\ldots,s$ are identical.  For $m=r+1,\ldots,s$, the remaining-rate
denominator under $\tau$ is smaller by exactly $d$ than the corresponding
denominator under $\sigma$.  Hence
\[
        \frac{\Pp(\tau)}{\Pp(\sigma)}
        =
        \prod_{m=r+1}^{s}\frac{D_m}{D_m-d}\ge1,
\]
where $D_m$ denotes the corresponding remaining-rate denominator under
$\sigma$.  Summing over the paired orderings gives
$\pi_k\ge\pi_i$.
\end{proof}

\begin{corollary}[Uniform lower bound]\label{cor:app-uniform-lower}
For every $p\in\Delta_N^\circ$,
\[
        H_C(p)\ge\frac{C}{N}.
\]
\end{corollary}

\begin{proof}
Using \eqref{eq:app-sums},
\[
        H_C(p)-\frac{C}{N}
        =
        \sum_k\left(p_k-\frac1N\right)\pi_k(p)
        =
        \frac1N\sum_{a<b}(p_a-p_b)(\pi_a(p)-\pi_b(p)).
\]
Every summand is nonnegative by Lemma~\ref{lem:app-mono}.
\end{proof}

\begin{theorem}[Jacobian master identity]\label{thm:app-master}
Let $N\ge2$, $1\le C<N$, and let $q\in\Delta_N^\circ$ be nonuniform.  Put
$\lambda(\theta)=u+\theta(q-u)$ and
$\mathcal H(\theta)=H_C(\lambda(\theta))$.  Then, for $0<\theta\le1$,
\begin{equation}\label{eq:app-master}
        \mathcal H'(\theta)
        =
        \frac{1}{\theta}\left(H_C(\lambda(\theta))-\frac{C}{N}\right)
        +
        \frac{1}{N\theta}
        \sum_{1\le i<k\le N}
        G_{ik}(\lambda(\theta))
        (\lambda_i(\theta)-\lambda_k(\theta))^2.
\end{equation}
In particular $\mathcal H'(\theta)>0$ for $0<\theta\le1$.
\end{theorem}

\begin{proof}
Write $\lambda=\lambda(\theta)$ and
$\dot\lambda_i=q_i-1/N$.  By the chain rule and \eqref{eq:app-sums},
\[
        \mathcal H'(\theta)
        =
        \sum_k\dot\lambda_k\pi_k
        +
        \sum_k\lambda_k\sum_i
        \frac{\partial\pi_k}{\partial\lambda_i}\dot\lambda_i
        =:T_1+T_2.
\]
Since $\dot\lambda_k=(\lambda_k-1/N)/\theta$, the first term is
\[
        T_1
        =
        \frac1\theta\sum_k\left(\lambda_k-\frac1N\right)\pi_k
        =
        \frac1\theta\left(H_C(\lambda)-\frac{C}{N}\right).
\]
For the second term, use Lemma~\ref{lem:app-sens} and
\eqref{eq:app-diagonal}:
\begin{align*}
        T_2
        &=
        \sum_i\lambda_i\dot\lambda_i\sum_{k\ne i}\lambda_kG_{ik}
        -
        \sum_{i\ne k}\lambda_k^2\dot\lambda_iG_{ik} \\
        &=
        \sum_{i\ne k}G_{ik}\lambda_k\dot\lambda_i(\lambda_i-\lambda_k).
\end{align*}
Symmetrizing over $i\leftrightarrow k$ gives
\[
        T_2
        =
        \sum_{i<k}G_{ik}(\lambda_i-\lambda_k)
        (\lambda_k\dot\lambda_i-\lambda_i\dot\lambda_k).
\]
Along the radial path,
\[
        \lambda_k\dot\lambda_i-\lambda_i\dot\lambda_k
        =
        \frac1\theta
        \left[\lambda_k\left(\lambda_i-\frac1N\right)
              -\lambda_i\left(\lambda_k-\frac1N\right)\right]
        =
        \frac{\lambda_i-\lambda_k}{N\theta}.
\]
Thus
\[
        T_2
        =
        \frac{1}{N\theta}\sum_{i<k}G_{ik}(\lambda_i-\lambda_k)^2.
\]
Adding $T_1$ gives \eqref{eq:app-master}.  The first term is
nonnegative by Corollary~\ref{cor:app-uniform-lower}.  The second term is
strictly positive because $G_{ik}>0$ for all pairs and $\lambda(\theta)$
is nonuniform when $0<\theta\le1$.  Hence $\mathcal H'(\theta)>0$.
\end{proof}

\begin{remark}[Relation to the main kernel]
The appendix proof yields a different positive decomposition of the same
radial derivative.  It splits the derivative into a nonnegative uniform
lower-bound term and a strictly positive occupancy-sensitivity square.
The main proof instead packages the whole derivative into the explicit
subset kernel $K_{ab}^{(C)}$ in \eqref{eq:kernel-def}.  The two views are
complementary: the subset proof gives the cleanest global certificate,
while the Jacobian proof explains the local mechanism.
\end{remark}


\begin{thebibliography}{12}

\bibitem{Berthet2016}
C. Berthet,
Identity of King and Flajolet \& al.\ formulae for LRU miss rate exact computation,
\emph{arXiv:1607.01283}, 2016.

\bibitem{CheTungWang2002}
H. Che, Y. Tung and Z. Wang,
Hierarchical Web caching systems: modeling, design and experimental results,
\emph{IEEE Journal on Selected Areas in Communications} \textbf{20} (2002), no.~7, 1305--1314.
\doi{10.1109/JSAC.2002.801752}

\bibitem{Fagin1977}
R. Fagin,
Asymptotic miss ratios over independent references,
\emph{Journal of Computer and System Sciences} \textbf{14} (1977), no.~2, 222--250.
\doi{10.1016/S0022-0000(77)80014-7}

\bibitem{FaginPrice1978}
R. Fagin and T. G. Price,
Efficient calculation of expected miss ratios in the independent reference model,
\emph{SIAM Journal on Computing} \textbf{7} (1978), no.~3, 288--297.
\doi{10.1137/0207025}

\bibitem{FillHolst1996}
J. A. Fill and L. Holst,
On the distribution of search cost for the move-to-front rule,
\emph{Random Structures \& Algorithms} \textbf{8} (1996), no.~3, 179--186.
\doi{10.1002/(SICI)1098-2418(199605)8:3<179::AID-RSA2>3.0.CO;2-V}

\bibitem{FlajoletGardyThimonier1992}
P. Flajolet, D. Gardy and L. Thimonier,
Birthday paradox, coupon collectors, caching algorithms and self-organizing search,
\emph{Discrete Applied Mathematics} \textbf{39} (1992), no.~3, 207--229.
\doi{10.1016/0166-218X(92)90177-C}

\bibitem{Hildebrand1999}
M. Hildebrand,
On a conjecture of Fill and Holst involving the move-to-front rule and cache faults,
\emph{Probability in the Engineering and Informational Sciences} \textbf{13} (1999), no.~3, 377--385.
\doi{10.1017/S0269964899133084}

\bibitem{Jelenkovic1999}
P. R. Jelenkovi\'c,
Asymptotic approximation of the move-to-front search cost distribution and least-recently-used caching fault probabilities,
\emph{The Annals of Applied Probability} \textbf{9} (1999), no.~2, 430--464.
\doi{10.1214/aoap/1029962750}

\bibitem{King1971}
W. Frank King III,
Analysis of demand paging algorithms,
in \emph{Information Processing 71: Proceedings of IFIP Congress 71},
Ljubljana, Yugoslavia, North-Holland, Amsterdam, 1972, pp.~485--490.

\bibitem{MakowskiVanichpun2005}
A. M. Makowski and S. Vanichpun,
Comparing locality of reference---some folk theorems for the miss rate and the output of caches,
in \emph{Performance Evaluation and Planning Methods for the Next Generation Internet},
A. Girard, B. Sans\`o and F. J. V\'azquez-Abad (eds.), Springer, Boston, MA, 2005, pp.~333--365.
\doi{10.1007/0-387-25551-6_13}

\bibitem{MarshallOlkinArnold2011}
A. W. Marshall, I. Olkin and B. C. Arnold,
\emph{Inequalities: Theory of Majorization and Its Applications}, second edition,
Springer, New York, 2011.
\doi{10.1007/978-0-387-68276-1}

\bibitem{VanichpunMakowski2004}
S. Vanichpun and A. M. Makowski,
Comparing strength of locality of reference: popularity, majorization, and some folk theorems,
in \emph{Proceedings of IEEE INFOCOM 2004}, vol.~2, Hong Kong, 2004, pp.~838--849.
\doi{10.1109/INFCOM.2004.1356972}

\end{thebibliography}
\end{document}